\documentclass[a4paper,12pt]{article}

\usepackage[english]{babel}
\usepackage{amssymb,amsfonts,amsmath,amsthm,mathtext}
\usepackage{cite,enumerate,float,indentfirst}
\usepackage[pdftex]{hyperref}
\usepackage{graphics}

\usepackage[centerlast,small]{caption2}
\title{Viro-Zvonilov inequalities for flexible curves on an almost complex four-dimensional manifold}
\author{Zvonilov V.I.\thanks{The author's work was done on a subject of the State assignment (N 0729-2020-0055).}}
\begin{document}
	
%
\newtheorem{Theorem}{Theorem}
\newtheorem{Proposition}{Proposition}
\newtheorem{Lemma}{Lemma}
\newtheorem{Corollary}{Corollary}
\newtheorem{Remark}{Remark}
\newtheorem{Definition}{Definition}
\newtheorem{Example}{Example}
	\maketitle	
	
	\begin{abstract} 
		The restrictions on the topology of nonsingular plane projective real algebraic curves of odd degree, obtained by O. Viro and the author in the paper published in the early 90s, are extended to flexible curves lying on an almost complex four-dimensional manifold. Some examples of real algebraic surfaces and real curves on them prove the sharpness of the obtained inequalities.
		In addition, it is proved that a compact Lie group smooth action can be lifted to a cyclic branched covering space $ \tilde{X} $ over a closed four-dimensional manifold, and a  sufficient condition for $ H_1(\tilde{X})=0 $ was found.
	\end{abstract}
	
%

\newcommand{\R}{\mathbb{R}}
\newcommand{\Cb}{\mathbb{C}}
\newcommand{\Z}{\mathbb{Z}}
\newcommand{\N}{\mathds{N}}
\newcommand{\ti}{\tilde}

 \section{Introduction}
It is well known that the first part of Hilbert's 16th problem is devoted to the  study of the topology of real algebraic varieties, including restrictions on the topology and constructions of the varieties.
This paper is about restrictions on the topology of real algebraic curves on a surface.
The obtained result extends  inequalities for a real nonsingular plane projective curve of odd degree, see \cite{VZ}, to the case of a flexible curve lying on an almost complex four-dimensional manifold. The inequality for the number of nonempty ovals of a plane curve of odd degree were formulated (without a proof) in \cite{R78}, inequality (11), and in \cite{V}, Theorems 3.10 and 3.11. The proof was published in 1993 in \cite{VZ}.

The problem of extending our old inequalities arose from M. Manzaroli's paper \cite{M} where the author used B\'{e}zout-type restrictions on the topology of real algebraic curves on del Pezzo surfaces.
The importance of the extended inequalities is that they give the restrictions not only for algebraic  but for so called \emph{flexible} curves introduced by O.Ya. Viro in \cite{V}. There is no B\'{e}zout-type restriction for flexible curves.

The scheme of the proof of the extended inequalities is the same as that of \cite[\S 1]{VZ} and use a cyclic branched covering space $ \tilde{X} $ over a closed four-dimensional manifold, the proof being essentially based on the results of \cite[\S 2]{VZ}. In Section \ref{def} the notion of flexible curve on an almost complex four-dimensional manifold with an anti-holomorphic involution is introduced. The extended inequalities are formulated in Section \ref{main}. The results of Sections \ref{bc} and \ref{lift} are of independent interest: in Subsection \ref{b1} a  sufficient condition for $ H_1(\tilde{X})=0 $ is found, in Section \ref{lift}  it is proved that a compact Lie group smooth action can be lifted to $ \tilde{X} $. Subsection \ref{acs} shows how to lift an almost complex structure to $ \tilde{X} $. The proofs of the main results are in Section \ref{p}. In Section \ref{Ap} some examples of real algebraic surfaces and real curves on them prove the sharpness of the obtained inequalities.
\section{Definitions}\label{def}
\subsection{Flexible curves}
Let $X$ be a smooth oriented connected closed four-dimensional  manifold with $H_1(X)=0$, and $ A $ be its oriented connected closed two-dimensional submanifold realizing a nonzero class  $\xi\in H_2(X) $. 
Assume that there exists a smooth involution $c:X\rightarrow X $ with non-empty fixed-point set $X^c $, and that $ A $ is invariant under $c$. Let $ A^c=A\cap X^c $.  
  By \cite[Ch. VI, Corollary 2.5]{B}) the set $ X^c $ is a smooth submanifold. Let $B$ be the union of the components of $ X^c $ that intersect $A$. Suppose that $ X $ is almost complex along $ B $, i.e. 
there is a smooth linear operator field $ J_x:T_xX\rightarrow T_xX $ for $ x\in B $  such that $ J^2=-\mathrm{id}$, and the involution $ c $ is anti-holomorphic with respect the almost complex structure: $ T_xX=T_xB\oplus JT_xB, d_xc(v_1+v_2)=v_1-v_2 $. Then $ \dim B=2 $.

 Following Viro \cite[\S 1]{V}, we call $ A $ a \emph{flexible curve}  of genus $ g $ if 
\renewcommand{\labelenumi}{(\theenumi)}
\renewcommand{\theenumi}{\roman{enumi}}
\begin{enumerate}
		\item  the genus of the surface $ A $ is equal to $ g $;
	\item the field of its tangent planes on $ A^c=A\cap B $  
	can be deformed, in the class of planes tangent to $ X $ and invariant under  $c$, to the field of complex lines tangent to $ A^c $. 
\end{enumerate}

Clearly, a real algebraic curve on a real algebraic surface is flexible.
    
A flexible curve
$A$ is of \emph{type}
	I or II if  $A\setminus A^c $ is disconnected or, correspondingly, connected. In the first case $ A^c $ divides $ A $  into two connected parts $ A^+, A^- $. An orientation of $ A $ induces orientations of $ A^+, A^- $ that, in turn, define  two opposite \emph{complex orientations} on $ A^c=\partial A^+=\partial A^- $.

\subsection{Membranes}
Let  $ B_1,\ldots,B_{k} $ be the 
orientable components of  $ B\setminus A^c $ with some orientations, 
 $ b_i\in H_2(B,A^c;\mathbb{Z}_p) $ be the class determined by $ B_i $, and
 $ k^+,k^0,k^- $ be the numbers of $ B_i $ with positive, zero, and negative Euler characteristics, which are called, correspondingly, \emph{elliptic, parabolic, and hyperbolic} components.
 
 Let $ Y_1,\ldots,Y_{r} $ be the 
 orientable connected components of $ B $. By $y_i\in H_2(B;\mathbb{Z}_p) $ denote the class determined by the component $ Y_i $.
\section{Main results}\label{main}
Let $ m>1 $ be the maximal odd number, such that 
$ m\,|\,\xi $, 
$ h=p^{\alpha}\,|\,m $ be the maximal power of a prime number  $ p $ that is a divisor of $ m $, and  $ \xi^2 $ be the self-intersection number of $ \xi\in H_2(X)$.

\begin{Theorem} [A bound on the number of hyperbolic components]\label{1}
	\begin{equation}\label{i1}
		k^-\leq \dfrac{b_2(X)+\sigma(X)}{2}+g-\dfrac{\xi^2(m^2-1)}{4m^2}, 
	\end{equation}
where $ b_2(X) $ and $ \sigma(X) $ are the second Betti number and the signature of $ X $. 
\end{Theorem}
Put $\rho=\dim\ker(\mathrm{in}_*:H_2(B;\mathbb{Z}_p)\rightarrow H_2(X;\mathbb{Z}_p))$, 	
and $ \delta=1$ or $ 0 $ if $ A $ is of type I or II.
\begin{Theorem} [A bound on the number of non-elliptic components]\label{2} If $ \pi_1(X\setminus A) $ is an abelian group then
	\begin{equation}\label{i2}
		k^-+k^0\leq \dfrac{b_2(X)+\sigma(X)}{2}+g-\dfrac{\xi^2(h^2-1)}{4h^2}+\rho+\delta.  
	\end{equation}
\end{Theorem}

	\begin{Theorem} [Extremal property of inequality (\ref{2})]\label{3}
	\begin{enumerate}
		
		\item \label{ext1} if equality holds in (\ref{2}) with $ \delta=1 $ then there are $ x_1,\ldots,x_k\in\mathbb{Z}_p $ such that the boundary homomorphism  $H_2(B,A^c;\mathbb{Z}_p)\rightarrow H_1(A^c;\mathbb{Z}_p)$ takes  $ x=\sum x_jb_j $ to the fundamental class $ [A^c] $ of the curve $ A^c$ endowed with a complex orientation, and $ \chi(B_j)=0 $ for all $ j $ with $ x_j\neq0 $.
		\item\label{ext0} 
		if equality holds in (\ref{2}) with $ \delta=0 $ and $ \rho>0 $ then there are $ x_1,\ldots,x_r\in\mathbb{Z}_p $ such that $\sum x_jy_j\in\ker\mathrm{in}_*$, and $ \chi(Y_j)=0 $ for all $ j $ with $ x_j\neq0 $. 
	\end{enumerate}

\end{Theorem}

According to \cite{N} for a nonsingular ample algebraic curve $ A $ on a simply connected projective surface $ X $ the group $ \pi_1(X\setminus A) $ is  abelian.
\begin{Remark}\label{rho}
	$ \rho $ is at most the number of $ Y_1,\ldots,Y_{r} $ with $ \chi(Y_j)\equiv0\, \mathrm{mod}\, p$.
\end{Remark} 
 \section{A branched covering}\label{bc} 

In this section we consider a more general case of a smooth oriented connected closed four-dimensional  manifold $ X $ with $H_1(X)=0$ and without involution.
Let  $ A $ be its oriented connected closed two-dimensional submanifold realizing a nonzero class  $\xi\in H_2(X) $.
By $ n $ denote the largest integer dividing $ \xi $, by $ q>1$ a divisor of $ n $, and by $ \mu$ the fraction $ n/q $.  According to Rokhlin \cite[\S 2]{R71}, there is a $ q $-sheeted cyclic covering $\nu: \tilde{X}\rightarrow X  $ branched over $ A $. The restriction $ \nu|_{\nu^{-1}(A)} $ is a diffeomorphism that identifies  $ A $ and $ \nu^{-1}(A) $.
\subsection{First homology of a cyclic branched covering space}\label{b1}
\begin{Theorem}\label{ab}
	 If $ \pi_1(X\setminus A) $ is an abelian group then $ H_1(\tilde{X})=0 $. In this case $ \pi_1(X\setminus A)=\Z_q$.
\end{Theorem}  
\begin{proof}
We use Rokhlin's construction of the covering $ \nu $ (see \cite[\S 2]{R71}).

Put  
$ \Pi=\pi_1(X\setminus A)=H_1(X\setminus A) $.  By $ _{\mu}\Pi $ denote the set of $ a\in\Pi $ such that $ a^{\mu}=1 $.  
The restriction $\nu: \tilde{X}\setminus A\rightarrow X\setminus A$ is an unbranched covering defined by a subgroup $_{\mu}\Pi= \nu_*\pi_1(\tilde{X}\setminus A)
\subset \pi_1(X\setminus A)$.

Let $ T $, $ \tilde{T} $ be  closed tubular neighborhoods of $ A $ in $ X $, $ \tilde{X} $, and $ U $, $ \tilde{U} $ their closed complements. The projections $pr:\partial T\rightarrow A $ and $\widetilde{pr}: \partial\tilde{T}\rightarrow A $ are $ SO(2) $-bundles with  circles $ C $ and $\tilde{C}$ as the fibers. Therefore, $ H_1(T,\partial T)=H_1(\tilde{T},\partial \tilde{T})=0 $ (see, e.g., \cite[Theorem 5.7.10]{S}), and so the inclusion homomorphisms $ H_1(\partial T)\rightarrow H_1(T) $, $ H_1(\partial \tilde{T})\rightarrow H_1(\tilde{T}) $ are epimorphisms. By virtue of the exactness of the additive sequences of triads $  X, 
T,  U  $ and $ \tilde{X}, \tilde{T}, \tilde{U} $ the equalities $ H_1(X)=0 $, $ H_1(\tilde{X})=0 $ holds iff the inclusion homomorphisms $ \beta:H_1(\partial T)=H_1(\partial U)\rightarrow H_1(U) $, $\tilde{\beta}: H_1(\partial \tilde{T})=H_1(\partial \tilde{U})\rightarrow H_1(\tilde{U}) $ are epimorphisms.
According to \cite[n. 2.2, 2.3]{R71}, the embedding homomorphism $\gamma: H_1(C)\rightarrow H_1(U)\cong H_1(X\setminus A)\cong \Z_n$ is an epimorphism, $ \nu_*:H_1(\tilde{C})\rightarrow H_1(C)$ is the miultiplication by $ q $, and $ \nu_*:H_1(\tilde{U})\cong_{\mu}\Pi\cong \Z_{\mu}\rightarrow H_1(U)\cong \Z_n$ is a monomorphism. Then $\tilde{\gamma}: H_1(\tilde{C})\rightarrow H_1(\tilde{U})$, equaled to $\nu_*^{-1}\circ\gamma\circ\nu_*$, is an epimorphism,  
and we have $ H_1(\tilde{X})=0 $.  	

\end{proof}

\section{Lifting a group action to a cyclic branched covering space}\label{lift}
 
\begin{Theorem}\label{G}
Let $ G $ be a compact Lie group (not necessarily connected) with a smooth action on $ X $. Suppose that there is a fixed point $ x_0 $ of this action. Choose a point $ \tilde{x}_0\in\tilde{X} $ with $ \nu(\tilde{x}_0)=x_0$. Then there exists a (unique) covering $ G $-action on $ \tilde{X}$ with a fix point $ \tilde{x}_0 $.  
\end{Theorem}  
\begin{proof}
	
According to \cite[Ch. I, Theorem 9.2]{B},   there exists a (unique) covering $ G $-action on $ \tilde{X}\setminus A$ with a fix point $ \tilde{x}_0 $ iff the subgroup $ \nu_*(\pi_1(\tilde{X}\setminus A, \tilde{x}_0))=_{\mu}\Pi $ of the group $ \pi_1(X\setminus A, x_0)=\Pi $ is invariant under $ G $-action on $ \Pi $. 
 An  
element $ a$ is in $_{\mu}\Pi $ iff $ a^{\mu}\in [\Pi,\Pi] $. It is clear that $ g_*[\Pi,\Pi]=[\Pi,\Pi] $ for any $ g\in G $. Thus, $ g_*(a^{\mu})=(g_*a)^{\mu}\in [\Pi,\Pi]$. Hence, $ g_*(a)\in$ $ _{\mu}\Pi $ and $g_*(_{\mu}\Pi)=$ $_{\mu}\Pi$.

By  \cite[Ch. VI, Theorem 2.6]{B}	the obtained $ G $-action on $ \tilde{X}\setminus A$ can be extended  to $ \tilde{X} $ by continuity so that $ A\subset \tilde{X} $ is $ G $-invariant. 
\end{proof}
\subsection{Lifting the involution}
	By Theorem \ref{G} the involution $ c:X\rightarrow X $ is covered by an involution $\tilde{c}$. According to \cite[Ch. VI, Theorems 2.2, 2.6]{B} there are invariant  closed tubular neighborhoods $ T $, $ \tilde{T} $ of $ A $ in $ X $, $ \tilde{X} $ with two-dimensional disks $ D $, $ \tilde{D} $ as the fibers  and with $ \nu(\tilde{T})=T $. Take a  point $ x_0\in\partial D\subset\partial T\cap X^c $ and one of its preimages $ y_0\in
	\partial \tilde{D} $. Considering that $c|_D$ is the reflection in the diameter with the end at $ x_0 $ and since $ q $ is odd, $\ti{c}|_{\ti{D}}$ is also the reflection in the diameter with the end at $ y_0 $  by uniqueness of $\tilde{c}$. 	
	By \cite[Ch. VI, Corollary 2.5]{B} the set $\tilde{X}^{\tilde{c}}=\mathrm{fix}\,\tilde{c}$ is a smooth submanifold of $ \tilde{X} $. Let $\tilde{B}=\tilde{X}^{\tilde{c}}\cap\nu^{-1}(B)$. It is clear that the restriction $\tilde{B}\rightarrow B $ of $ \nu $ is a diffeomorphism.

\subsection{Lifting the almost complex structure to a branched covering space}\label{acs}

Since $ X $ is almost complex along $ A^c\subset X^c $, using the definition of flexible curve, assume that the tangent plane at every point   of $ A^c $ is a complex line tangent to $A^c$.  Therefore the projection $pr: T\rightarrow A $ over $A^c$ is a $ U(1) $-bundle with the complex unit disk $ D $ as the fiber. The unbranched covering $\nu: \tilde{X}\setminus A\rightarrow X\setminus A$ allows to lift the almost complex structure of $ X\setminus A $ over $ B\setminus A $ to an almost complex structure of $ \tilde{X}\setminus A $ over $ \tilde{B}\setminus A $. Since the differential of $ \nu $ is zero on $ A^c$ the almost complex structure can be extended to $ \tilde{B}\supset A^c $ by continuity. It is clear that the involution $ \tilde{c} $ is anti-holomorphic with respect the almost complex structure along $ \tilde{B} $.  

\subsection{An automorphism of the covering and its invariant subspace}\label{M}
The automorphism group of the covering $ \nu $ is isomorphic to $ \Z_q $ and acts on the fibers of the normal bundle of the surface $ A\subset \tilde{X} $  like the group of rotations of the plane through angles that are multiples of $ 2\pi/q $. Let $ \tau $ be either of two generators of the group which acts like a rotation through $ 2\pi/q $.   

For $ \xi=\exp(\pi(q-1)\sqrt{-1}/q) $ consider a subspace $ M=\ker(\tau_*-\xi^{-1}\mathrm{id})\subset H_2(\tilde{X};\Cb)$, where $\mathrm{id}$ is the identity homomorphism. Let $ Q $ be the restriction to $ M $ of the Hermitian intersection form of $ \tilde{X} $, and $ \mathrm{sign}\,Q $ be its signature. From Rokhlin's calculations \cite[n. 5.4, 5.6]{R71} it follows that $ \dim M=b_2(X)+2g $, and $\mathrm{sign}\,Q=\sigma(X)-\frac{\xi^2(q^2-1)}{2q^2}$. 
\section{Proof of the main results}\label{p}
Note that in the case of a real algebraic curve the inequalities (\ref{1}),  (\ref{2}) are also true for $ m=1 $. Indeed, it is well known that $ \dfrac{b_2(X)+\sigma(X)}{2}\geq1 $ for a compact complex surface with $H_1(X)=0$ (see, e.g., \cite[Ch. IV, Theorem (2.7)]{BHPV}). Thus, for $ m=1 $ the inequalities (\ref{1}),  (\ref{2}) follow from the Harnack inequality, and (\ref{2}) is not an equality for $ \delta=1 $.  

There are only two cases below when $ q=m $ (in Subsection \ref{p1} )
 and $ q=h $ (in the rest of the text). For simplicity, the index $ q $ is omited in the objects depending on $ q $.

 \subsection{Covering classes over the membranes}\label{mem}
 Denote by $ \tilde{B}_i$ the connected component of $\tilde{B}$ such that $\nu(\tilde{B}_i)= B_i$, $i=1,\ldots,k$. The orientation of $ B_i $ defines an orientation of $ \tilde{B}_i $. 
  Let $ \gamma_{ij}\in H_2(\tilde{X};\Cb) $ be the class defined by the surface $ \mathrm{Cl}(\tilde{B}_i\cup\tau^j\tilde{B}_i) $ oriented in accordance with the orientation of $ \tilde{B}_i $. For $ i=1,\ldots,k $ put $ \beta_i=\sum_{j=1}^{q-1}\xi^j\gamma_{ij} $. Simple calculation shows that $ \beta_i\in M $. 
  
     Since each component $ Y_i $ of $ B $ is the closure of a union of some $ B_j $,  there is a class $\zeta_i=\sum\beta_j\in M $.
 \subsection{The intersection numbers of the covering classes}
 Since the almost complex stricture allows to apply Arnold's scheme \cite{A} for calculation the intersection numbers using tangent vector fields on surfaces, the calculation is exactly the same as in \cite[n. 1.3]{VZ}, i.e., $ \beta_i\circ\beta_r=0 $ when $ i\neq r $ and
 \begin{equation}\label{4}
 \beta_i^2=-\chi(B_i)q.	
 \end{equation} 
 \subsection{Proof of Theorem \ref{1}}\label{p1}
 Here we put $ q=m$. Consider a subset of $ \{\beta_1,\ldots,\beta_{k}\} $ that generate the subspace $ M^+ $ of $ M $ on which $ Q\geq0$. Since the classes $ \beta_i $ are pairwise orthogonal the set consists of  $ \beta_i $ with $ \beta_i^2>0 $, and these classes are linearly independent. By (\ref{4}) the number of such classes is $ k^- $, which is the dimension of $ M^+ $. It does not exceed $ (\dim M+\mathrm{sign}\,Q )/2 $. Using the formulas from \S \ref{M} we obtain Theorem~\ref{1}. 
 \subsection{The rank of the system of the covering classes}\label{rank}
  Let $ q=h $ everywhere in what follows.
 \begin{Lemma}\label{in}
 	 Let 	$in_*:H_2(B,A^c;\Z_p)\rightarrow H_2(X,A;\Z_p)$ be the inclusion homomorphism. Then $\mathrm{rk}(\beta_1,\ldots,\beta_{k})\geq\mathrm{rk}(in_*(b_1),\ldots,in_*(b_{k})) $ 
 \end{Lemma} 
\begin{proof}
 Clearly, the classes $\{\beta_1,\ldots,\beta_{k}\} $ belong to the image of the homomorphism  $\lambda_*:H_2(\tilde{X};\Z[\xi])\rightarrow H_2(\tilde{X};\Cb)$ induced by the inclusion $\lambda:\Z[\xi]\rightarrow \Cb$. By Theorem \ref{ab} we have $ H_1(\tilde{X})=0 $. Using the duality and the universal coefficient theorem we get $ H_3(\tilde{X})=0 $ and $ \mathrm{Tors} H_2(\tilde{X})=0 $, and consequently $ H_3(\tilde{X};\Cb/\Z[\xi])=0 $. Therefore $ \lambda_* $ is a monomorphism. Hence  
 $\mathrm{rk}(\beta_1,\ldots,\beta_{k})=\mathrm{rk}(\lambda_*^{-1}(\beta_1),\ldots,\lambda_*^{-1}(\beta_{k})) $.

 By $ \mu $ denote the ring homomorphism $ \Z[\xi]\rightarrow\Z_p $ that takes $ \xi $ to $ 1 $. It is not trivial since $ q$ is a prime power (see \cite[\S 1.5]{VZ}). Let $ \beta'_i=\mu_*\lambda_*^{-1}(\beta_i)\in H_2(\tilde{X};\Z_p) $.
 If $\lambda_{1}\beta_1+\ldots+\lambda_{k}\beta_{k}=0  $ is a nontrivial linear relation with the coefficients $ \lambda_i\in\Z[\xi] $ that are coprime in $ \Z[\xi] $ then $\mu(\lambda_{1})\beta'_1+\ldots+\mu(\lambda_{k})\beta'_{k}=0  $ is a nontrivial linear relation (see \cite[\S 1.5]{VZ}). 
 Thus $\mathrm{rk}(\beta_1,\ldots,\beta_{k})\geq\mathrm{rk}(\beta'_1,\ldots,\beta'_{k})) $.  By \cite[Lemma 2.3.A]{VZ} the homomorphism $ \tilde{\alpha}_2:H_2(X,A;\Z_p)\rightarrow H_2(\tilde{X};\Z_p) $ from the Smith homology sequence is a monomorphism and its composition $\tilde{\alpha}_2\circ in_*$   takes the membrane class $ b_i $ to $ \beta'_i $. Since the membrane classes form a basis of $ H_2(B,A^c;\Z_p) $ we have $\mathrm{rk}(\beta'_1,\ldots,\beta'_{k}))=k-\dim\ker\,in_*=\mathrm{rk}(in_*(b_1),\ldots,in_*(b_{k}))$. Thus $\mathrm{rk}(\beta_1,\ldots,\beta_{k})\geq\mathrm{rk}(in_*(b_1),\ldots,in_*(b_{k})) $.  
	\end{proof} 
\begin{Lemma}\label{ker}
For the inclusion homomorphism 	$ \dim\ker\,in_*\leq \rho+\delta $ where $ \rho$ and $ \delta $ are defined in \S \ref{main}.
\end{Lemma}
\begin{proof}
From the exactness of the homology sequence of the pair $ B, A^c $ it follows that $rel^c: H_2(B;\Z_p)\rightarrow H_2(B,A^c;\Z_p) $ is a monomorphism. So the basis $y_1,\dots, y_r $ for $ H_2(B;\Z_p) $ can be  enlarged
to a basis $rel^c y_1,\dots,rel^cy_r,b_{r+1},\dots,b_k$ for $ H_2(B,A^c;\Z_p)$ and $\partial^c b_{r+1},\dots,\partial^c b_k  $ is a basis for $ \mathrm{im}\,\partial^c\subset H_1(A^c) $. Since $ \xi=0\,\mathrm{mod}\,p$  the inclusion homomorphism 	$H_2(A;\Z_p)\rightarrow H_2(X;\Z_p)$ is zero. Thus we have a segment of the homology exact sequence of the pair $ X, A $ 
\begin{equation}\label{seq}
	0\rightarrow H_2(X)\stackrel{rel}\rightarrow H_2(X,A)\stackrel{\partial}\rightarrow H_1(A)\rightarrow0 
\end{equation}
with $ rel $ being a monomorphism. Therefore $ \dim\ker\,in_*=\dim(\mathrm{im}\partial^c\cap\ker\,in_1)+\dim\ker\,in_2$ where $in_1:H_1(A^c;\Z_p)\rightarrow H_1(A;\Z_p))$ and $in_2:H_2(B;\Z_p)\rightarrow H_2(X;\Z_p))$ are the inclusion homomorphisms.
It is well known that $ \delta=\dim\ker\,in_1$. Then $ \dim\ker\,in_*\leq \rho+\delta $.  
\end{proof}
 Let $ d_i\in H_1(A;\Z_p) $ denote the class realizable by the boundary of $ B_i $. 
  
 The following corollary is evident.
 \begin{Corollary}\label{I}
 	In the notation of the previous proof,  $ d_i=in_1\partial^c(b_i)=\partial\, in_*(b_i) $. If $ d_1,\ldots,d_k $ are linear dependent then $ A $ is of type I and the equality $ in_1[A^c]=0$ (see (\ref{ext1}) of Theorem \ref{3}) gives a nontrivial linear relation for $ d_1,\ldots,d_k $.
 \end{Corollary}
 Put $ z_i=in_2(y_i) $. 
\begin{Lemma}\label{B}	
	For a nontrivial linear relation \begin{equation}\label{l1}
		\lambda_{1}\beta_1+\ldots+\lambda_{k}\beta_{k}=0 
	\end{equation}  with $ \lambda_i\in\Z[\xi] $ being coprime in $ \Z[\xi] $,  either $ \mu(\lambda_1)d_1+\ldots+\mu(\lambda_k)d_k=0$ is  a nontrivial linear relation and $ A $ is of type I, or (\ref{l1})  can be rewritten as
\begin{equation}\label{l2}
	\mu_1\zeta_1+\ldots+\mu_r\zeta_r=0
\end{equation} (see \S \ref{mem}) with some $\mu_i\in\Z[\xi] $ being coprime in $ \Z[\xi] $ and  $\mu(\mu_1)z_1+\ldots+\mu(\mu_r)z_r=0$ is  a nontrivial linear relation.
\end{Lemma}
\begin{proof}
	As above, $\mu(\lambda_{1})\beta'_1+\ldots+\mu(\lambda_{k})\beta'_{k}=0  $ is a nontrivial linear relation.  Since $ d_i=\partial\tilde{\alpha}_2^{-1}(\beta'_i) $ then  $ \mu(\lambda_1)d_1+\ldots+\mu(\lambda_k)d_k=0$. If it is  a nontrivial linear relation then $ A $ is of type I by Corollary~\ref{I}, and the first case is proved. 
	
	We can rewrite (\ref{l1}) as $ \mu_1\zeta_1+\ldots+\mu_r\zeta_r+\mu_{r+1}\beta_{r+1}+\ldots+\mu_{k}\beta_{k}=0 $, getting also $ \mu(\mu_1)\zeta'_1+\ldots+\mu(\mu_r)\zeta'_r+\mu(\mu_{r+1})\beta'_{r+1}+\ldots+\mu(\mu_{k})\beta'_{k}=0 $ where $ \zeta'_i=\mu_*\lambda_*^{-1}(\zeta_i) $, because $rel^c y_1,\dots,rel^cy_r,b_{r+1},\dots,b_k$ is a basis for $ H_2(B,A^c;\Z_p)$  and $ \zeta'_i=\tilde{\alpha}_2in_*rel^c(y_i)=\tilde{\alpha}_2rel(z_i)$ where $\tilde{\alpha}_2$ and $ rel  $ are monomorphisms (see Proof of Lemma~\ref{ker}). If the relation $ \mu(\lambda_1)d_1+\ldots+\mu(\lambda_k)d_k=0$ is trivial then $ \mu(\mu_{r+1})=\ldots=\mu(\mu_{k})=0 $
	by the exactness of the sequence (\ref{seq}). Thus  
$\mu(\mu_1)z_1+\ldots+\mu(\mu_r)z_r=0$	 is  a nontrivial linear relation.  
\end{proof}
 \subsection{Proof of Theorem \ref{2}}\label{p2}
 Consider a maximal subset $ L $ of the set $\{\beta_1,\ldots,\beta_{k}\} $ that generates a space whose intersection form is nonnegative. Since the classes $ \beta_i $ are pairwise orthogonal, $ L $ consists of $ \beta_i $ with $ \beta_i^2\geq0 $. By (\ref{4}) the number of such classes is $ k^0+k^- $. Denote by $ l $ the dimension of $ L $. Since $ \beta_i\in M $ and according to \cite[\S 4.2]{R71} the form $ Q $ is nondegenerate, $ l\leq(\dim M+\mathrm{sign}\, Q)/2 $. On the other hand, $k^0+k^-\leq l+\rho+\delta$ by \ref{rank}. Combining these two inequalities and the information about $ M $ from Subsection \ref{M} we obtain (\ref{i2}).
 \subsection{Proof of Theorem \ref{3}}
 Suppose that equality is attained in (\ref{i2}).  
 Use the arguments of Subsections \ref{rank}, \ref{p2}.  
 
 If $ \delta=1 $ then $\mathrm{rk}(\beta_1,\ldots,\beta_{k})=k-\rho-1$, the curve $ A $ is of type I and $ A^c $, endowed  with a complex orientation, gives    a generator $ [A^c] $ of $ \ker in_1\subset \mathrm{im}\partial^c$.   Therefore there is a class $ x=\sum x_jb_j\in H_2(B,A^c;\Z_p) $ with $ \partial^c x=[A^c] $. Besides there is a nontrivial linear relation~(\ref{l1})
 	   such that, by Corollary~\ref{I} and Lemma~\ref{B}, $in_1[A^c]= \mu(\lambda_1)d_1+\ldots+\mu(\lambda_k)d_k=0$ is  a nontrivial linear relation. Again, by Corollary~\ref{I} we have $d_i=in_1\partial^c(b_i) $. So the numbers $ x_1,\ldots,x_k $ are proportional to $\mu(\lambda_1),\ldots,\mu(\lambda_k)$. Therefore if $ x_i\neq0 $ then $ \lambda_i\neq0 $. Multiplying (\ref{l1}) by $ \beta_i $ and using that $ \beta_1,\ldots,\beta_{k} $ are pairwise orthogonal with their squares determined by (\ref{4}), we deduce that $ \chi(B_i)=0 $ for all $ i $ with $ x_i\neq0 $.

If $ \delta=0 $ then $\mathrm{rk}(\beta_1,\ldots,\beta_{k})=k-\rho$ and the curve $ A $ is of type II. If $\rho>0 $, by Lemma \ref{B}, there is a nontrivial linear relation (\ref{l2}) such that  $\mu(\mu_1)z_1+\ldots+\mu(\mu_r)z_r=0$ is also a nontrivial linear relation. 
Thus for $x_j=\mu(\mu_j)$ we have $ \sum x_jy_j\in\ker in_* $ and if $x_i=\mu(\mu_i)\neq0 $ then $ \mu_i\neq0 $. Multiplying (\ref{l2}) by $ \zeta_i $ and taking into account that $ \zeta_1,\ldots,\zeta_{k} $ are, clearly, pairwise orthogonal with $ \zeta_i^2=-\chi(Y_i)h $ (a consequence of (\ref{4})), we deduce that $ \chi(Y_i)=0 $ for all $ i $ with $ x_i\neq0 $.
\subsection{Proof of Remark \ref{rho}}
Suppose that $ \lambda_{1}z_1+\ldots+\lambda_{r}z_{r}=0 $ with some $ \lambda_{i}\in \Z_p $. Multiplying this equality by $ z_i $ and taking into account that $ z_1,\ldots,z_{r} $ are, clearly, pairwise orthogonal with $ z_i^2\equiv-\chi(Y_i)\, \mathrm{mod}\, p $ (see \cite[Lemma 6]{A}), we deduce that $ \chi(Y_i)\equiv0\, \mathrm{mod}\, p$ for all $ i $ with $ \lambda_i\neq 0\, \mathrm{mod}\, p $. Thus $ \ker\,in_* $ is a subspace of the linear span of  $ Y_i $ with $ \chi(Y_i)\equiv0\, \mathrm{mod}\, p$, and the inequality is proved. 
  \section{Application to real algebraic curves on surfaces}\label{Ap}
  Here we apply the inequalities (\ref{i1}), (\ref{i2}) to nonsingular real algebraic curves on some simply connected surfaces.
  \subsection{Plane curves}
  	Recall (see \cite{VZ}) that for a plane curve of odd degree $ m $ the inequalities (\ref{i1}), (\ref{i2}) take the form
  	
  $k^-\leq\frac{(m-3)^2}{4}$, 
    
  $k^-+k^0\leq\frac{(m-3)^2}{4}+\frac{m^2-h^2}{4h^2}+\delta $ 
  
  where $ h=p^{\alpha}\,|\,m $ for a prime $ p $ and the maximal $ \alpha $. They are sharp for $ m=3 $ (a curve with an oval) and for $ m=5 $ (a curve with a nest of two ovals). 
\subsection{Curves on a quadric}  
  A curve of bidegree $ (a,b) $  on a hyperboloid or ellipsoid is ample if $ a>0$, $b>0 $ (see \cite[Ch. V, 1.10.1]{H}). For such a curve the inequalities (\ref{i1}),~(\ref{i2}) take the form
  
  $k^-\leq \dfrac{ab(m^2+1)}{2m^2}-a-b+2,$
  
  $ k^-+k^0\leq \dfrac{ab(h^2+1)}{2h^2}-a-b+3+\delta$
  
  with an odd $ m\,|\,\mathrm{gcd}(a,b) $ and $ h=p^{\alpha}\,|\,m $ for a prime $ p $ and the maximal $ \alpha $. Both inequalities are sharp  for curves of bidegree $ (3,3) $ on a hyperboloid. The first one becomes the equality for a curve with $ l $ ovals, $ 1\leq l\leq4 $, (and with a real component of degree (1,1)); the second one is sharp for a small perturbation of the intersection of the hyperboloid with three parallel planes.
\subsection{Curves on a Hirzebruch surface}  
   Let $\pi: \Sigma_e\rightarrow P^1 $ be a Hirzebruch surface with the zero section $ Y,\, Y^2=e, $ and a fiber $ F $. A nonsingular real curve  
   $A=aY+bf\subset\Sigma_e$ of bidegree $ (a,b) $ is ample if $ a>0$, $b>0 $ (see \cite[Ch. V, 2.18]{H}). For such a curve
    the inequalities (\ref{i1}), (\ref{i2}) take the form
    
  $k^-\leq \dfrac{(ae+2b)(am^2-2m^2+a)}{4m^2}-a+2,$
  
  $ k^-+k^0\leq \dfrac{(ae+2b)(ah^2-2h^2+a)}{4h^2}-a+2+\rho+\delta$ 
  
  with an odd $ m\,|\,\mathrm{gcd}(a,b) $ and $ h=p^{\alpha}\,|\,m $ for a prime $ p $ and the maximal $ \alpha $. 
      Both inequalities are not sharp for $ e>0 $.
 \subsection{Curves on a Del Pezzo surface}
 	A del Pezzo surface $ X $ is a complete non-singular surface with ample anti-canonical class $ -K_X=c_1(X)$, and
 	 $d=K_X^2$   the degree of  $ X $. It is well known that $ 1\leq d\leq 9$, $ X=P^2$ for $d=9$ and $ X=P^1\times P^1$ for $ d=8 $.
 	 If $ A\in nc_1(X)\subset X $ is a nonsingular real curve then
 	  the inequalities (\ref{i1}), (\ref{i2}) take the form
 	
 	$k^-\leq \dfrac{(n-1)^2d}{4}+\dfrac{(n^2-m^2)d}{4m^2}+2,$
 	
 	$ k^-+k^0\leq \dfrac{(n-1)^2d}{4}+\dfrac{(n^2-h^2)d}{4h^2}+2+\rho+\delta$ 
 	
 	with an odd $ m\,|\,n $ and $ h=p^{\alpha}\,|\,m $ for a prime $ p $ and the maximal $ \alpha $.   If $X^c$ is a torus then $ \rho \leq1 $ else $ \rho=0 $.
 	
 	Let $ X  $ be a del Pezzo surface
 	of degree $d=2$ that is a two-sheeted covering of the projective plane branched over a curve $ C_4 $ of degree $ 4 $. Suppose that $ A $ is a lifting to $ X $ of the plane curve $ C_3 $ of degree $ 3 $. 
 	
 Suppose that $ C_4 $ consists of four (empty) ovals, $ O_4 $ is one of them,  the odd branch (i.e. one-sided component) $ J_3 $ of $ C_3 $ and $ O_4 $ have the intersection with the code $ 123a9458bc76 $ that means the following. The points of $ O_4\cap J_3 $ are numbered by $ 1,\ldots,9,a,b,c $ in their order along $ O_4 $. The code of the intersection is the order of the points along $ J_3 $ (see Figure \ref{DPII} where $ J_3 $ is a line and the oval of $ C_3 $ is inside $ O_4 $). From the construction of this arrangement (see \cite[Figure 5.12]{K} or \cite[Figure 53]{O}) it is clear that the pair of points $ 1,2 $ or $ b,c $ can be deleted so that the real curve $ A^c\subset B\subset X^c $, with $ B $ being a sphere, consists of a nest of four ovals with three empty ovals inside the inner oval of the nest. The number of ovals of $ A $ is $ 7 $, one less than maximal, thus it is of type II and $ \delta=0 $. By Remark \ref{rho} we have $\rho=0 $. Therefore the second inequality (with $n=h=3$): $ k^-+k^0\leq 4$  is sharp.
 	\begin{figure}[h]
 		\begin{center}
 			\scalebox{0.9}{\includegraphics{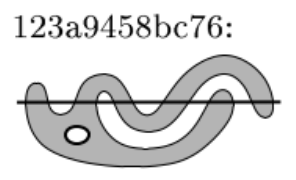}}\\
 			\end{center}
 		\caption{}
 		\label{DPII}
 	\end{figure} 
 
 	Another example of the curve $ A $ is obtained from $ C_4 $ consisting of a nest of two ovals (see Figure \ref{DP}). For such a curve the second inequality (with $n=h=3$, $\rho=1 $ and $ \delta=1 $): $ k^-+k^0\leq 6$  is sharp.  
 	
 	\begin{figure}[h]
 		\begin{center}
 			\scalebox{0.9}{\includegraphics{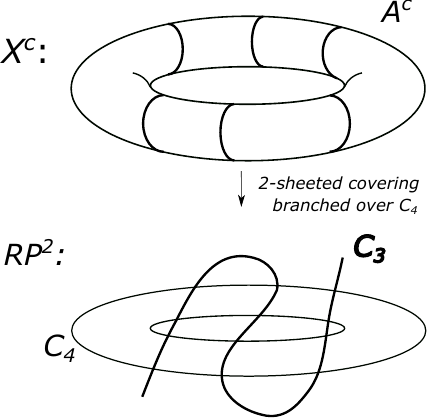}}\\
 			\end{center}
 		    \caption{}
 		        \label{DP}
 	\end{figure} 
	The author is grateful to S.Yu.~Orevkov for paying author's attention at  M. Manzaroli's paper \cite{M}, for stating the problem and for useful discussions.


\end{document}